\documentclass[runningheads,envcountsamme,a4paper]{llncs}

\usepackage[margin=3.9cm]{geometry}

\usepackage{float}
\usepackage{amssymb}

\usepackage{caption}

\usepackage{amsmath}

\usepackage{thmtools}

\usepackage{thm-restate}
\usepackage{subcaption}
\usepackage{caption}

\usepackage[T1]{fontenc}
\usepackage{graphicx}
\begin{document}
\title{Burning a binary tree and its generalization}
\titlerunning{Burning a binary tree and its generalization}
%

\author{Sandip Das \and
Sk Samim Islam \and
Ritam M Mitra \and
Sanchita Paul\thanks{Corresponding author}}

\authorrunning{Das, Islam, Mitra, and Paul}
%

\institute{Indian Statistical Institute, Kolkata, India}


%
\maketitle              
\begin{abstract}
 Graph burning is a graph process that models the spread of
social contagion. Initially, all the vertices of a graph $G$ are unburnt. At each step, an unburnt vertex is put on fire and the fire from burnt vertices of the previous step spreads to their adjacent unburnt vertices. This process continues till all the vertices are burnt. The burning number $b(G)$ of the graph $G$ is the minimum number of steps required to burn all the vertices in the graph. The burning number conjecture by Bonato et al. states that for a connected graph $G$ of order $n$, its burning number $b(G) \leq \lceil \sqrt{n} \rceil$. It is easy to observe that in order to burn a graph it is enough to burn its spanning tree. Hence it suffices to prove that for any tree $T$ of order $n$, its burning number $b(T) \leq \lceil \sqrt{n} \rceil$ where $T$ is the spanning tree of $G$. It was proved in 2018 that $b(T) \leq \lceil \sqrt{n + n_2 + 1/4} +1/2 \rceil$ for a tree $T$ where $n_2$ is the number of degree $2$ vertices in $T$. In this paper, we provide an algorithm to burn a tree and we improve the existing bound using this algorithm. We prove that $b(T)\leq \lceil \sqrt{n + n_2 + 8}\rceil -1$ which is an improved bound for $n\geq 50$. 
We also provide an algorithm to burn some subclasses of the binary tree and prove the burning number conjecture for the same.
\end{abstract}

\section{Introduction}
\label{sec:introduction}

Graph burning is a process that captures the spread of social contagion and was introduced by Bonato, Janssen, and Roshanbin \cite{bonato2014burning} in 2014.  The way ideas and feelings spread among people on social networks is a hot topic in social network research. For instance, check out these studies \cite{banerjee2014epidemic,domingos2001mining,kempe2003maximizing,kempe2005influential,mossel2007submodularity,richardson2002mining}. A recent study focused on how emotions spread on Facebook \cite{kramer2014experimental}. What's interesting is that it found that the structure of the network itself is crucial. Surprisingly, you don't need to meet someone in person or see their body language for emotions  to spread. Instead, people in the network share these feelings with their friends or followers, and it keeps spreading over time. Now, here's the question: If you wanted to make this emotional spread happen as quickly as possible across the whole network, who should you start with, and in what order? To address this problem, we study graph burning. We can think that each person or profile is a vertex of a graph. Two people can share their feelings in one step if there is an edge between them. We say a vertex is burnt if the feeling is spread to the vertex. 

 Below, we describe the process of burning a simple graph $G(V,E)$ where $|V|=n$. Graph burning consists of discrete steps. After each step, each vertex is either burnt or unburnt, once a vertex is burnt it remains burnt till the end. Initially, all the vertices are unburnt. In the first step, we burn a vertex. At each subsequent step, two things happen: first, one new unburned vertex is chosen to burn; second, the fire spreads from each burnt vertex of the previous step to its neighboring vertices and burns the unburned vertices from the neighbors. If a vertex is burned,
then it remains in that state until the end of the process.  The process ends when all the vertices are burnt. The \textit{burning number}, denoted by $b(G)$, is the minimum number of steps taken for this process to end.
The burning problem asks, given a graph G and an integer $k \geq 2$, whether $b(G) \leq k$. 
An intuitive way to look at this process is to cover the vertices of the graph $G$ by $k$ balls of radius $ 0, 1, \hdots,k - 1 $, placed at appropriate vertices such that $b(G)$ is minimized.
A ball of radius $r$ placed at a vertex $v$ can cover vertices that are at a distance of at most $r$ from $v$.
For example, it is straightforward to see that $b(K_n) = 2$. However, even for a relatively simple graph such as the path $P_n$ on $n$ nodes, computing the burning number is more complex; in fact,
$b(P_n) = \lceil \sqrt{n}\rceil$ as \cite{bonato2014burning} proved. Suppose that in the process of burning a graph $G$, we eventually burn the whole graph $G$ in $k$ steps, and for each $i$, $1 \leq i \leq k$, we denote the node where we set the fire in the $i$th step by $x_i$. We call such a node simply a source of fire. The sequence $(x_1, x_2, \hdots ,x_k)$ is called a burning sequence for $G$. 

Graph burning can be likened to the way viruses propagate within populations. When left uncontrolled, a limited number of infected individuals within a population can lead to the widespread transmission of the virus. Similarly, one can draw an analogy to a forest fire, which, when left to burn without intervention, can devastate entire forests across a vast expanse. 

Bessy et al. \cite{bessy2017burning} showed that  the Graph Burning problem is
NP-complete when restricted to the trees of maximum degree three. This implies that the burning graph problem is NP-complete for binary trees, chordal graphs, bipartite graphs, planar graphs, and disconnected graphs. Moreover, they showed the NP-completeness of the
burning problem even for trees with a structure as simple as spider graphs, and also for disconnected
graphs such as path-forests. Furthermore, there is a polynomial time
approximation algorithm with approximation factor $3$ for general graphs. Hence, determining the precise burning number of a graph is currently not the primary focus. Now, a question can be raised: Is there any tight  relation between the burning number and the number of vertices of a graph?  
Bonato,  Janssen, and Roshanbin \cite{bonato2016burn} proposed the Burning Number Conjecture that relates these two numbers of a graph. 

\smallskip

\noindent \textbf{Burning Number Conjecture (BNC):}\label{bnc}  Let $G$ be a connected graph of order $n$, then $b(G)\leq \lceil \sqrt{n}\rceil$.

\smallskip

\noindent In order to burn a connected graph, it is sufficient to burn its spanning trees \cite{bonato2016burn}. Therefore, in the study of graph burning, 
much focus has been spent on the trees.  Note that, if $T$ is a spanning tree of a graph $G$, then  $b(G)\leq b(T)$, since the number of steps required to burn  $T$ is at least  the number of steps required to burn  $G$. Therefore, if the BNC holds for all the trees, then it definitely holds for all the graphs. Now the main challenge is to solve the BNC for the trees.  
Till now the BNC has been solved for some subclasses of trees including Spider \cite{bonato2019bounds,das2018burning}, Double spider \cite{tan2023burnability}, Caterpillars \cite{liu2020burning}.


\vspace{0.3cm}
\noindent\textbf{Our Contribution:} In this paper, we study the graph burning problem on several types of binary trees. 


A {\em perfect binary tree} is a special type of binary tree in which all the leaf nodes are at the same level and all internal nodes have exactly two children. In Section \ref{sec:Perfect Binary Tree}, we provide a relation between the height and the burning number of a perfect binary tree:

\begin{restatable}{theorem}{Perfecto}\label{perfect}
Let $T$ be a perfect binary tree of height $h$, then $b(T)=h+1$.
\end{restatable} 

\noindent Using  Theorem \ref{perfect}, we show that the BNC is true for the perfect binary tree. 

\smallskip

Section \ref{sec:Complete Binary Tree} is based on the study of the graph burning problem on the complete binary tree. A {\em complete binary tree} is a binary tree in which every level, except the last, is completely filled, and all nodes in the last level are filled from as left as possible. We provide a relation between the height and the burning number of a  complete binary tree:

\begin{restatable}{theorem}{pft}\label{pft}
Let $T$ be a binary tree of height $h$ which needs exactly $1$ more leaf to be perfect. Then $b(T)=h+1$.
\end{restatable}

\begin{restatable}{theorem}{cbtt}\label{cbtt}
    Let $T$ be a complete binary tree of height $h$ which does not have at least $2$ leaves in it's last level. Then $b(T)=h$.
\end{restatable}

\noindent Using Theorem \ref{pft} and Theorem \ref{cbtt}, we show that the BNC is true for the complete binary tree.
\smallskip

A {\em full binary tree} is a binary tree in which all of the nodes have either $0$ or $2$ children. In other words, except for the leaf nodes, all other nodes have two children, and it is not necessary for all leaf nodes to be at the same level. In section \ref{sec:Full Binary Tree}, we show that the burning number of a full binary tree which is not perfect is upper bounded by its height. We state that in the following theorem:

\begin{restatable}{theorem}{FBTNP}\label{FBTNP}
Let $T$ be a full binary tree of height $h$, which is not perfect, then $b(T)\leq h$.
\end{restatable}

\noindent In the same section, we provide an algorithm to burn the full binary tree which is not perfect, that yields the following Theorem which says the BNC is true for the same:

\begin{restatable}{theorem}{firstalgo}\label{firstalgo}
Let $T$ be a full binary tree of $n$ vertices, which is not perfect. Then we have $b(T)\leq \lceil \sqrt{n}\rceil $.
\end{restatable}

\noindent Moreover, we provide a tighter bound of the burning number for the full binary tree which is not perfect having the number of vertices more than $17$. We state that in the following theorem:

\begin{restatable}{theorem}{secondalgo}\label{secondalgo}
Let $T$ be a full binary tree on $n$ vertices, which is not perfect, then $b(T)\leq \lceil \sqrt{n+9}\rceil -1$. 
\end{restatable}

  In 2018, Bessy et al. proved \cite{bessy2018bounds}  that $b(T) \leq \lceil \sqrt{n + n_2 + 1/4} +1/2 \rceil$ for a tree $T$ where $n_2$ is the number of degree $2$ vertices in $T$. Using Theorem 5, we improve this bound for the trees having the number of vertices at least $50$. The improved bound is the following: 


\begin{restatable}{theorem}{generaltree}\label{generaltree}
    Let $T$ be a tree of order n and $n_2$ be the number of degree 2 vertices. Then $b(T)\leq \lceil \sqrt{n + n_2 + 8}\rceil -1$.
\end{restatable}



\section{Preliminaries}\label{sec:Preliminaries}
Since every connected graph is spanned by a tree, burning the spanning tree is sufficient to burn the connected graph's vertices. {\em Binary trees} are the most commonly known trees, which have several applications in data structure and algorithms \cite{cormen2001introduction}. Binary trees have the characteristic feature that they are rooted trees, and every vertex has at most two children. The {\em depth} of a node in a tree is the number of edges present in the path from the root node to that node. For a binary tree $T$, any two vertices are said to be at the same {\em level} if they are at the same depth from the root. The root is considered to be at the 1st level, and any other vertex that is at level $i$ is said to be $i-1$ distant from the root.
Any non-leaf node is said to be an {\em internal node}. We call the internal nodes that are the parents of leaves as {\em parent-leaves}.
 We call one node to be the {\em sibling} of other if they have a common parent. The {\em height} of a node is the number of edges present in the path connecting that node to a leaf node. For any other definition, we follow the standard notation of West \cite{West}. 
\vspace{0.5em}

 Given a positive integer $k$, the $k$-{\em th closed neighbourhood} of $v$ in a tree $T=(V,E)$ is defined by $N_{k}[v]=\{u\in V: d(u,v)\leq k\}$.





\section{Perfect Binary Tree}\label{sec:Perfect Binary Tree}
A {\em perfect binary tree} is a special type of binary tree in which all the leaf nodes are at the same level and all internal nodes have exactly two children. In simple terms, this means that all leaf nodes are at the maximum depth of the tree. It is to be noted that a perfect binary tree of height $h$ has $2^h$ number of leaf nodes.

\Perfecto*




\begin{proof}
 $T$ has total $h+1$ levels, since $T$ has height $h$. 
    Suppose we set the fire at the root of $T$ which is at the $1$st level. At the $t$-th step, the fire completes the burning of the $t$-th level of $T$. Therefore, the burning of the whole $T$ will be completed after $h+1$ steps, since $T$ has total $h+1$ levels. Therefore,   $b(T)\leq h+1$. 

    Now we will show that $b(T)\geq h+1$. We prove this by contradiction. Suppose there is a burning sequence $S=(x_{1},x_{2},\hdots,x_{\alpha})$, where $\alpha < h+1$. After completing the burning of $T$ using this burning sequence, the fire that was set on the vertex $x_j$ can spread to the vertices of the ball $N_{\alpha-j}[x_j]$, for all $j$ where $1\leq j\leq \alpha$. Note that, $N_{\alpha-j}[x_j]$ can contain at most $2^{\alpha-j}$ leaves of $T$. Therefore, the total number of leaves that can be burned using the burning sequence $S$ is at most $\sum_{j=1}^{\alpha} 2^{\alpha-j}=2^{\alpha}-1$. But the total number of leaves of $T$ is $2^{h}$. Now, $2^{h}\geq 2^\alpha > 2^\alpha-1$. Therefore, all the leaves of $T$ can not be burned by $S$, which is a contradiction. Therefore, $b(T)\geq h+1$.  \hfill $\square$ 
\end{proof}

A perfect binary tree $T$ of height $h$ is of order $n=2^{h+1}-1$, therefore, we have $b(T)=h+1=\log_{2} (n+1)\leq \lceil \sqrt{n}\rceil$ for $n\geq 19$. This validates the burning number conjecture for this graph class. 

\section{Complete Binary Tree}\label{sec:Complete Binary Tree}

A {\em complete binary tree} is a binary tree in which every level, except the last, is completely filled, and all nodes in the last level are filled from as left as possible. The number of nodes (leaves) at the last level, ranges from $1$ to $2^h-1$, where $h$ is the height of the tree. 

\pft*
\begin{proof}
The fact that $b(T)\leq h + 1$ can be determined by placing the source of fire at the root of $T$ at the first step of burning. 
\vspace{0.3em}

We will now show that $T$ cannot be burnt in $h$ steps. On the contrary, let $(x_{1},x_{2},\hdots,\\
x_{h})$ be an arbitrary burning sequence of $T$, i.e., $b(T)\leq h$.  After completing the burning of $T$ using this burning sequence, the fire that was set on the vertex $x_i$ can spread to the vertices of the ball $N_{h-i}[x_i]$, for all $i$ where $1\leq i\leq h$. Now we show that for $x_{i}$, the total number of leaves that can be burnt from the fire spread from $x_{i}$, i.e., which are in $N_{h-i}[x_{i}]$ is at most $2^{h-i}$ in number and maximality is attained only when $x_i$ is at $i+1$ th level. If $x_i$ is at $j$ th level where $1\leq j\leq i$, then any leaf (say $l$) that is the descendant of $x_{i}$ remains unburnt from the fire spread from $x_{i}$ since, $d(x_{i},l)=h+1-j\geq h+1-i> (h-i)$. Therefore, it is necessary to place $x_{i}$ at $j (\geq i+1)$ th level to burn the leaves, and furthermore, maximality can only be attained when $j=i+1$. Note that $x_{h}$ cannot burn any more vertex except itself. Hence, the total number of leaves and parent-leaves that can be burnt by all the $x_{i}$'s up to $h$ th step of burning is at max $ \sum_{i=1}^{h-1} (2^{h-i}+2^{h-i-1})+1=2^{h}+2^{h-1}-2$. Again, since there is only one leaf is missing at the $h+1$ th level, the total number of leaves and parent leaves in $T$ is $(2^{h}-1)+2^{h-1}$. Therefore, at least $1$ vertex (leaf/parent-leaf) is left unburnt after $h$ th step. Therefore $b(T)=h+1$. \hfill $\square$ 
\end{proof}

\cbtt*
\begin{proof}
First, we show that $h$ steps are sufficient to burn $T$ in this case, i.e., $b(T)\leq h$. Let $c$ be the root of $T$ and $T_1, T_2$ be two subtrees of $T$ rooted at the two children $c_{1},c_{2}$ of $c$. Since $T$ is a complete binary tree where at least two leaves are absent at its last level, at least one among $T_{1}, T_{2}$, let's say $T_{1}$ is of height $h-1$, while $T_{2}$ is of height $h-2$ or $h-1$.  Let $u_{1},u_{2},\hdots,u_{h-1}$ be the vertices on a branch of $T_{2}$ that does not contain any leaf from $h+1$ th level (see Figure \ref{cbt})
    satisfying $d_{T}(c,u_{i})=i$ and $x_{i}$ be the sibling node of $u_{i}$. Then $d(c,x_{i})=i$. We set $v_{1}=c_{1}, v_{i}=x_{i},2\leq i\leq h-1$ and $v_h=u_{h-1}$. Now consider the sequence $(v_{1}, v_{2},\hdots, v_{h-1}, v_{h})$.  As every subtree rooted at $v_{i}, 1\leq i\leq h-1$ is of height $h-i$ or $h-i-1$ and $v_{h}$ is a leaf node of $T_2$, we get $\bigcup\limits_{i=1}^{h} N_{h-i}[v_{i}]=V(T)$. Therefore, $(v_1,v_2,\hdots,v_h)$ becomes the burning sequence of $T$. Thus, up to $h$ steps $T$ must be burnt.

\vspace{0.3em}
Now we will show that $T$ cannot be burnt in $h-1$ steps. Let $T^{\prime}$ be the perfect binary subtree of $T$ having height $h-1$. If we assume on the contrary that $T$ can be burnt in $h-1$ steps, then $T^{\prime}$ should also be burnt in $h-1$ steps. But from Theorem \ref{perfect} we know that $b(T^{\prime})=h$, which contradicts our assumption. Hence $b(T)=h$.   \hfill $\square$ 
\end{proof}

\vspace{0.5mm}

\begin{figure}[htp]
\includegraphics[clip,width=\columnwidth]{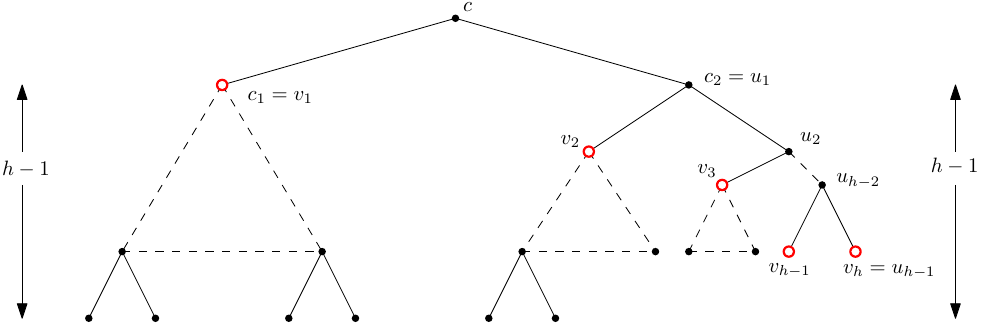}
\caption{Complete Binary Tree}\label{cbt}
\end{figure}

A complete binary tree of height $h$ that does not contain exactly one leaf in its last level is of order $ 2^{h+1}-2$. Hence $b(T)= h+1=\log_{2}(n+2)\leq \lceil\sqrt{n}\rceil$ when $n\geq 20$.  For the other complete binary trees, $n\geq 2^{h}$, therefore, $b(T)=h\leq \log_{2} n\leq \lceil\sqrt{n}\rceil$. Hence, the burning number conjecture holds true for this subclass.

\section{Full Binary Tree}\label{sec:Full Binary Tree}
  A {\em full binary tree} is a binary tree in which all of the nodes have either $0$ or $2$ children. In other words, except for the leaf nodes, all other nodes have two children, and it is not necessary for all leaf nodes to be at the same level.

\FBTNP*

\begin{proof}
Let $c$ be the root and $x_{1},x_{2}$ be the two children of $c$ in $T$. By induction hypothesis, we assume that the result is true for all such trees of height less than $h$. Now consider the subtrees $T_{1}, T_{2}$ of $T$ having $x_{1},x_{2}$ as their roots. It is easy to note that $T_{1}, T_{2}$ are two disjoint trees. Also, they cannot be perfect binary trees of height $h-1$ simultaneously, as $T$ is not perfect. Without loss of generality, let $T_1$ be any full binary tree of height $h-1$ and $T_2$ be a full binary tree of height $\leq h-1$ which is not perfect. Thus, using the algorithm stated in the proof of Theorem \ref{perfect}, $T_1$ can be burnt in $h$ steps by only placing fire at $x_1$ in the $1st$ step. Note that we do not need to set fire to any other vertex of $T_1$ to burn it in $h$ steps. Now, by the induction hypothesis, we know $b(T_2) \leq h-1$.  Also if $T_2$ is perfect and of height $\leq h-2$ then it can be burned upto $h-1$ steps by the algorithm stated in the proof of Theorem \ref{perfect}. Let the burning sequence to burn $T_2$ be $(y_1, \hdots ,y_{h-1})$. Thus, the burning sequence $(x_1, y_1, \hdots ,y_{h-1})$ is sufficient to burn $T$ upto $h$ steps. Therefore, $b(T)\leq h$. \hfill $\square$ 
\end{proof}

 \textit{ For convenience, we denote the class of full binary tree which is not perfect by the notation \textbf{FBTNP}(Full Binary Tree Not Perfect).}


\begin{proposition}\label{maximalfullbinary}
Given any positive integer $k$, there exists a maximal full binary tree $T=(V,E)$ which is not perfect satisfying $|V|=3(2^{k}-1)-2k$, that can be burnt in $k$ steps.
\end{proposition}

\begin{proof}

     \begin{figure}
         \centering
         \includegraphics[width=\textwidth]{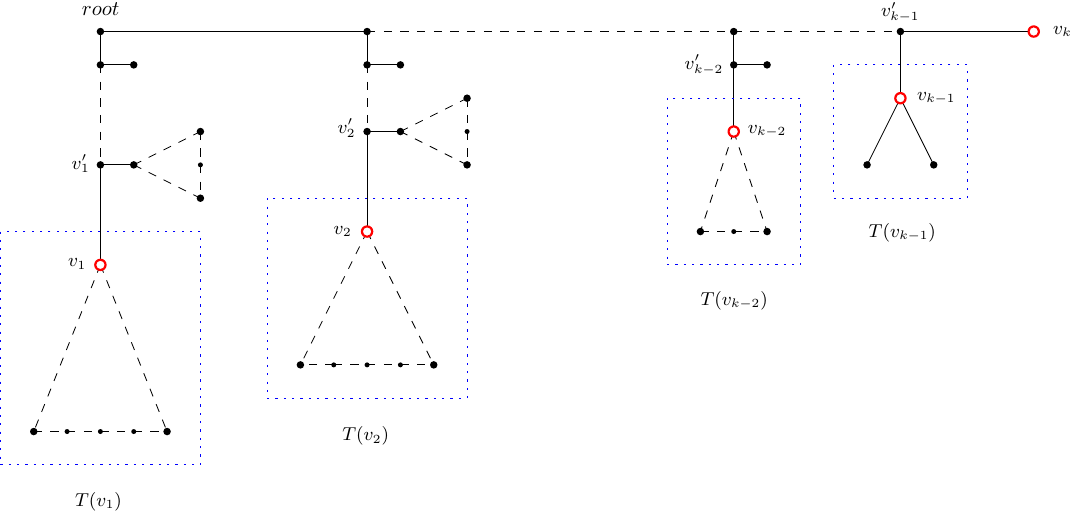}
         
         \caption{Maximal Burning} \label{MNb}
    \end{figure}



  It is sufficient to create a burning sequence $(v_{1},v_{2},\hdots,v_{k})$ in such a manner so that $N_{k-1}[v_{1}]\cup N_{k-2}[v_{2}]\cup \hdots \cup N_{0}[v_{k}]=V$ for some FBTNP tree $T$ and $|V|=3(2^{k}-1)-2k$. Maximality can be guaranteed whenever we are able to prevent all instances of double-burning in $T$, i.e., if $T$ can be constructed in such a way that $N_{k-i}[v_{i}]\cap N_{k-j}[v_{j}]=\emptyset$ for all $i\neq j$. 
\vspace{0.3em}

   First, we construct a perfect binary tree $T(v_{1})$ (say) of height $k-1$, rooted at the vertex $v_{1}$.  Let $v_{1}^{\prime}$ be the neighbor of $v_1$ apart from the two children of $v_1$ in $T(v_{1})$. Clearly, $v_{1}^{\prime} \notin V(T(v_{1}))$. Next, we create a perfect binary tree of height $k-2$ rooted at $v_{1}^{\prime}$, say $T(v_{1}^{\prime})$ disjoint from $T({v_{1}})$.
It is important to note that, $T(v_{1})$ and $T(v_{1}^{\prime})$ are two separate perfect binary trees and are connected by the edge  $v_1 v_{1}^{\prime}$. Let, $T_1 = T(v_{1}) \cup T(v_{1}^{\prime})$ be a subtree of $T$. It is easy to verify that $|N_{k-1}[v_{1}]|= |V(T_1)| =|V(T(v_1))|+|V(T({v_1}^{\prime}))|=(2^{k}-1)+(2^{k-1}-1)=2^{k}+2^{k-1}-2$ (see Figure \ref{MNb}). We construct the subtrees $T_2, T_3, \hdots T_k$ sequentially in the aforementioned procedure where $T_i = N_{k-i}[v_{i}]$ and $|T_i| = |N_{k-i}[v_{i}]| = 2^{(k-i)+1}+2^{k-i}-2$ for $2\leq i\leq k$. We create an edge between any two leaves of $T_i, T_j$ for all $i\neq j$. Thus we build $T = T_1 \cup T_2 \cup T_3 \hdots \cup T_k$.

  One can verify that $T$ is an FBTNP with the root as one of the leaves of $T(v_{1}^{\prime})$. Furthermore, the total number of vertices of $T$ is $(2^{k}+2^{k-1}-2)+\sum_{i=2}^{k} (2^{k-i+1}+2^{k-i}-2)=3 (2^{k}-1)-2k$. Since each of $T(v_i)$ and $T(v_{i}^{\prime})$ are perfect binary trees of height $k-i$ and $k-i-1$ respectively, and $v_{i}v_{i}^{\prime}\in E$, using Theorem \ref{perfect} we can burn the FBTNP $T$ in $k$ steps. The maximality of $T$ is ensured by the above construction.
  \hfill $\square$ 
\end{proof}

\subsection{Algorithm to burn an FBTNP}\label{Algo1}
\begin{figure}[ht]
     \centering
         \centering
         \includegraphics[width=\textwidth]{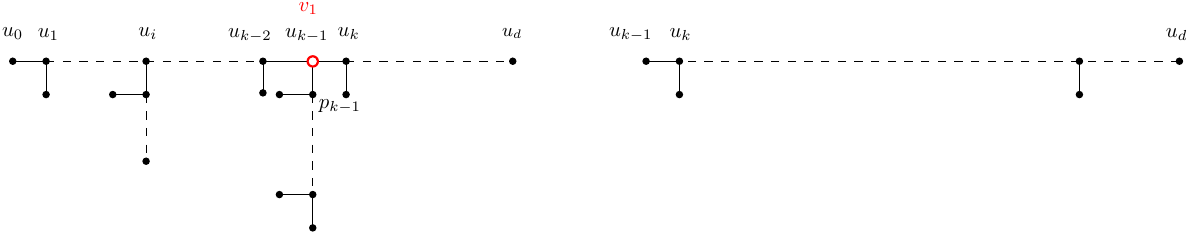}
         \label{}
     \hfill
         \centering
         \includegraphics[width=\textwidth]{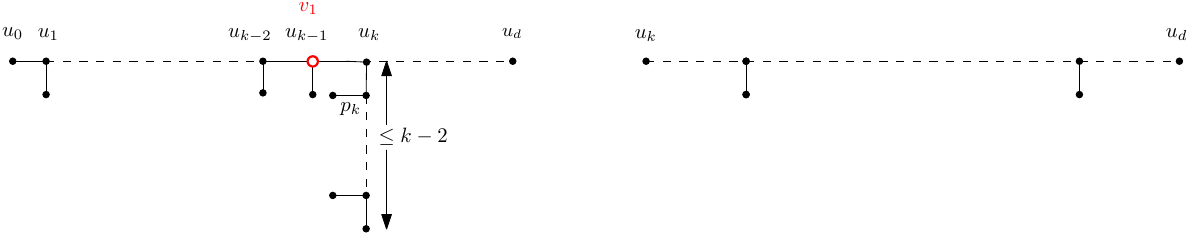}
         \label{}
     \hfill
         \centering
         \includegraphics[width=\textwidth]{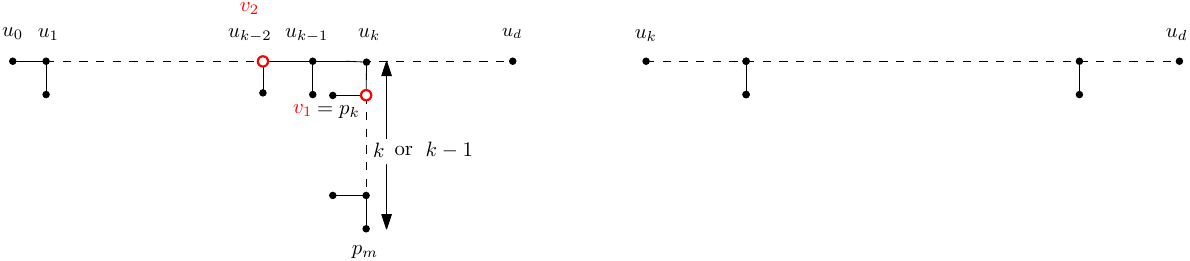}
         \label{}
     \caption{Cases in Full Binary Tree (case (a) at the top, case (b)(i) in the middle, case (b)(ii) in the bottom)}
\label{FBT}
\end{figure}

In this subsection, we propose an algorithm to burn an FBTNP which is recursive in nature. We construct a burning sequence $(v_1,\hdots,v_k)$ of an FBTNP tree $T$ of height $h$ for some $k\in \mathbb{N}$. Our idea is to burn a portion of the tree $T$ at each step by using one or two sources of fire. When $k<h$, we cut the portion from the main tree in such a way that the remaining graph becomes an FBTNP \footnote{except 
some trivial situations where after generating the sequence $(v_{1},\hdots,v_{k-2})$, the modified tree may not be an FBTNP. In these cases, it can be $P_2$ or $K_{1,3}$ and both can be burned by placing $v_{k-1}$ and $v_k$} itself with a smaller size. Otherwise, only a single source of fire spreading upto $k$ steps is sufficient to burn the whole tree. In the Theorem \ref{firstalgo}, we prove that for an FBTNP of order $n$, the value of $k=\lceil \sqrt{n}\rceil$.

\smallskip

\noindent \textbf{Algorithm: }  Let $T=(V,E)$ be an FBTNP having $P_{d+1}=(u_{0},u_{1},\hdots,u_{d})$ as the diametral path, where $d$ is the diameter and $h$ be the height of $T$.  Below we will construct a burning sequence $(v_1, v_2, \dots v_k)$  of $T$ for some $k \in \mathbb{N}$. For $0 \leq i \leq d $, we define $T(u_{i})$ to be the subtree rooted at $u_{i}$, formed by the branches of $T$ attached to the vertices $u_{0},u_{1},\hdots,u_{i}$ excluding the branch $(u_{i+1},\hdots,u_{d})$. Let $u_{c}$ be the root of $T$ and $p_{i}$ be the neighbour of $u_{i}\neq u_{c}$, where $p_{i} \neq u_{i-1}$ or $u_{i+1}$. Similarly, another subtree $T^{\prime} (u_{i})$ rooted at $u_{i}$ can be formed by the branches attached to the vertices $u_{i}, u_{i+1},\hdots,u_{d}$, excluding the branch $(u_{0},\hdots,u_{i-1})$ at $u_{i}$ in $T$.

\vspace{0.3em}
\noindent In order to burn the tree $T$ in $k$ steps, we consider the two circumstances.
\vspace{0.3em}

\noindent $1$) Let $k\geq h$. If the root $u_{c}$ lies on the diametral path $(u_{0}, \hdots, u_{c},\hdots, u_{k},\hdots,u_{d})$. Then either $u_{c}=u_{k-1}$ or it occurs left to $u_{k-1}$ along the path $P_{d+1}$. We place $v_{1}$ at $u_{c}$ to burn $T(u_{c})\cup T^{\prime}(u_{c})=T$ up to $k$ steps.

In the other situation, i.e., when the root $u_{c}$ is not on the diametral path $P_{d+1}$, then it must be on a branch attached to some $u_{l}$, where $0<l<k$ or $d-k<l<d$. Without loss of generality, we assume $0<l<k$ and $p_{m}$ be the maximum distance vertex on that branch measured from $u_{l}$.  We place $v_{1}$ at $u_{l}$ to burn $T(u_{l})\cup T^{\prime}(u_{l})=T$ up to $k$ steps.
\vspace{0.3em}

Choices of $v_{2},\hdots,v_{k}$ in this case, can be arbitrary.  

\vspace{0.3em}

\noindent $2$) Let $k<h$. We will apply the following technique to form the burning sequence $(v_{1},v_{2},\hdots,v_{k})$ of $T$. We start the algorithm from the maximum end of $P_{d+1}$, calculated from the root of $T$. If both of the endpoints of $P_{d+1}$ are of the same distance from the root, then we can start from any end. Without loss of generality, assume that $u_{0}$ is at maximum distance from the root (Refer to Figure \ref{FBT}).

\vspace{0.3em}
\noindent \textbf{Step I:}
\vspace{0.3em}

\noindent a) Let the number of vertices of the subtree $T(u_{k-2})$ be greater than equal $2k-1$. Then the order of $T(u_{k-2})\cup \{p_{k-1}\}$ becomes $\geq 2k$. One can observe that for this to happen there must exist at least one branch from a vertex $u_{i}$, $1\leq i\leq k-2$ having its height $\geq 2$ which is different from $P_{d+1}$. Now we place $v_{1}$ at $u_{k-1}$ in order to burn the vertices of $T(u_{k-2})\cup \{p_{k-1}\}$ upto $k$ steps. 

\vspace{0.3em}

 Next, we update the diametral path $P_{d+1}$ to $P_{d-k+2}=(u_{k-1},u_{k},\hdots,u_{d})$. It is important to note that we include $u_{k-1}$ in the updated diametral path so that the modified tree $T^{\prime}(u_{k})\cup \{u_{k-1}\}$ preserves all the features of the initial tree $T$. 

\vspace{0.5em}

\noindent b) Next, we consider the case when the order of the subtree $T(u_{k-2})$ is less than $2k-1$. Since $k<h$, no vertex $u_{i}$, $0\leq i\leq k$ can be the root of $T$. Therefore $|V(T(u_{k-2})|\geq 2k-3$ and hence $T(u_{k-1})\cup \{p_{k}\}$ is of order $\geq (2k-3)+3=2k$ as $T$ is an FBTNP. 

\vspace{0.3em}
The following situations may further occur.

\vspace{0.3em}
\noindent i) The branch attached to $u_{k}$ different from $P_{d+1}$ is of maximum height $k-2$. Then we place $v_{1}$ at $u_{k-1}$ to burn the vertices of $T(u_{k-1})\cup \{p_{k}\}$ upto $k$ steps. 

\vspace{0.3em}
\noindent ii) The branch attached to $u_{k}$ different from $P_{d+1}$ is of height $k-1$ or $k$. Let $p_m$ be the last vertex in the attached branch. Then $m = k$ or $k-1$. Apart from $p_{k}$, there must be at least $2k-2$ vertices in this branch. We place $v_{1}$ at $p_k$ and $v_{2}$ at $u_{k-2}$ so that up to $k$ th step the total number of vertices burnt due to the fire spread from the vertices $v_{1},v_{2}$ become $\geq 4k-4$. 

\vspace{0.5em}

After this, the diametral path $P_{d+1}$ of $T$ is updated, resulting in the new diametral path $P_{d-k+1}=(u_k,\hdots,u_d)$.  In order to ensure that the modified tree $T^{\prime}(u_{k+1})\cup \{u_{k}\}$ retains all the properties of the original tree $T$, we include the vertex $u_k$ in the updated diameter \footnote{in both of these cases, we have placed $v_{1}$ or $\{v_{1},v_{2}\}$  on the branches of $T$ in a suitable manner so that the modified tree remains a connected FBTNP and the remaining graph completely burns upto $k$ steps}.

\vspace{0.5em}

\noindent \textbf{Step II:} Repeat \textbf{Step I}.



\firstalgo*

\begin{proof}
Let $P_{d}=(u_{0},u_{1},\hdots,u_{d})$ be the diametral path of $T$. We follow the Algorithm \ref{Algo1} to burn $T$ in $k$ steps. We will use induction on $|V(T)|$ to prove the theorem.


\vspace{0.3em}
\noindent \textit{Induction Hypothesis :} Suppose the result is true for all FBTNP trees having order less than equal $t^{2}$, i.e., if $|V(T)|\leq t^{2}$ then we get $b(T)\leq t$.

\vspace{0.3em}

\noindent \textit{Inductive Step:} Consider an FBTNP tree $T$ of $m$ verices such that $t^2<m \leq (t+1)^{2}$. In order to prove the theorem, it is sufficient to show that $T$ can be burnt in $t+1$ steps. Assume $k=t+1$.

\vspace{0.3em}
First, we consider the situation when $k\geq h$. If $u_{c}$ is on $P_{d+1}$ then $h=d(u_{0},u_{c})<d(u_{c},u_{k})=k$. In the other situation, $d(u_{0},u_{c})=h<k=d(u_{0},u_{k})$ imply $d(u_{0},u_{l})<k$. Again, $d(u_{l},u_{d})<d(u_{c},u_{d})<h\leq k$. Furthermore, $d(u_{l},u_{m})<k$ follows from the definition of diametral path $P_{d+1}$. More specifically, $N_{k-1}[v_{1}]=V(T)$ as $v_{1}=u_{l}$. Therefore all the vertices of $T$ will get burnt up to $k=t+1$ th step due to the fire spread from $v_{1}$.

\vspace{0.3em}

 Next, we consider the case when $k<h$. For Case (a) and Case (b) (i), the total number of vertices burnt due to the fire spread from $v_{1}$ upto $t+1$ th step is $\geq 2(t+1) \geq 2t+2$. Moreover, since $P_{d+1}$ is the diametral path, any branch attached to $u_{k-1}$ is of maximum height $k-1$ and the branch attached to $u_{k}$ is of maximum height $k-2$ in case (b) (i). 
Therefore, after placing $v_{1}$ at $u_{k-1}$, the subtree $T(u_{k-1})$ in case (a) and the subtree $T(u_{k})$ in case (b)(i) will be completely burnt up to $k$ steps. 

\vspace{0.3em}
 One can observe that the order of the modified tree becomes less than equal to $ m-2t-2$ after updating the diameter for both of the above cases. Since $m\leq (t+1)^2$, the number of vertices  of the modified tree is $\leq (t+1)^2-2t-2=t^2-1<t^{2}$. Thus, the modified tree can be burnt in $t$ steps by the induction hypothesis.

\vspace{2mm}

 For Case (b) (ii), the total number of vertices burnt as a result of the fire spreading from $v_{1},v_{2}$ up to $t+1$ th step is $\geq 4(t+1)-4 \geq 4t$. Also, it is important to note that any branch attached to $u_{k}$ can be of maximum height $k$, as a result, the placement of $v_{1},v_{2}$ as described in the algorithm ensures the complete burnability of the subtree $T(u_{k})$ up to $k$ th step.  
Due to this burning, the order of the modified tree becomes less than equal $m-4t\leq (t+1)^2-4t\leq (t-1)^2$. Therefore, according to the induction hypothesis the modified tree can be burnt in $t-1$ steps. 


\vspace{0.3em}

\noindent \textit{Conclusion:} Thus, by induction the FBTNP $T$, having $|V(T)|\leq (t+1)^{2}$ can be burnt in $(t+1)$ steps. Hence we get $b(T)\leq \lceil \sqrt{m} \hspace{0.3em}\rceil$.  \hfill $\square$ 
\end{proof}


\subsection{Improved Bound Algorithm to burn an FBTNP}\label{Algo2}

\noindent \textbf{Algorithm :} Let $P_{d+1}=(u_{0},u_{1},\hdots,u_{d})$ be the diametral path where $d$ is the diameter and $h$ be the height of an FBTNP $T=(V,E)$. We consider the following cases, depending on which we develop a method to burn the tree $T$ in $k$ steps. The burning sequence of $T$ will be built as $(v_1, v_2, \hdots v_k)$. If $k\geq h$, then we follow the method to burn $T$ as described in Algorithm \ref{Algo1}. Now we will consider the situation when $k<h$. Let $L_{i}=(p_{i,1},p_{i,2},\hdots,p_{i,h^{\prime}})$ be the branch attached to $u_{i}$ different from $P_{d+1}$ where $1\leq i\leq d-1$ and $h^{\prime}=d(u_{i},p_{i,h^{\prime}})$ be the height of $L_{i}$ computed from $u_{i}$. To keep the notations simple, we denote the subtree rooted at  $p_{i,l}$, $1\leq l\leq h^{\prime}$ by considering all the branches which are attached to the path $(p_{i,l},\hdots,p_{i,h^{\prime}})$. Other notations will remain the same as in Algorithm \ref{Algo1}.

\vspace{0.6em}
\noindent \textbf{Step I:}
\vspace{0.3em}

\noindent \textbf{Case 1:} First, we consider the case where $T(u_{k-1})$ has order $\geq 2k+1$. Here are three situations that may happen. For this to happen, there must be a branch from some $u_i$, $1\leq i\leq k-1$ of height $\geq 2$ that is different from $P_{d+1}$.
\vspace{0.7em}

a) If the branch $L_{k-1}$ attached to $u_{k-1}$ is of height $\geq 2$, then order of $T(u_{k-1})\setminus \{u_{k-1}\}\geq 2k+2$. In this situation, we place $v_{1}$ at $u_{k-1}$ to burn $\geq 2k+2$ vertices up to $k$ steps. Next, we update the diametral path $P_{d+1}$ to $P_{d-k+2}=(u_{k-1},\hdots,u_{d})$. It is important to note that we include the vertex $u_{k-1}$ in the updated diametrical path so that the modified tree $T^{\prime}(u_{k})\cup \{u_{k-1}\}$ retains all properties of the original tree $T$.(Refer Figure \ref{IAcase1a})
\vspace{0.7em}
\begin{figure}[ht]
     \centering
    \begin{subfigure}[b]{0.45\textwidth}
         \centering
         \includegraphics[width=\textwidth]{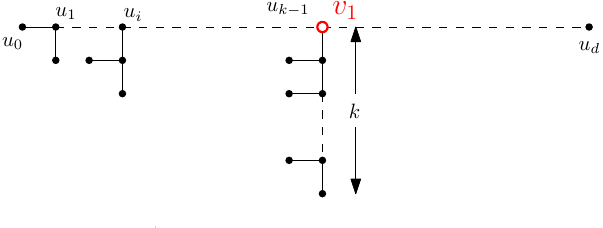}
         \caption{}
         \label{IAcase1a}
     \end{subfigure}
     \hfill
     \begin{subfigure}[b]{0.45\textwidth}
         \centering
         \includegraphics[width=\textwidth]{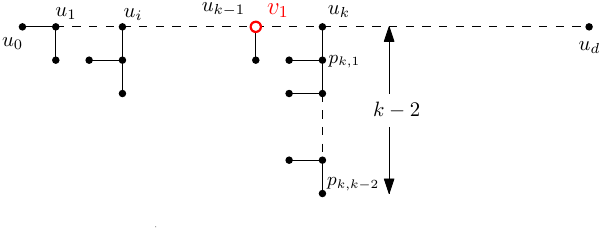}
         \caption{}
         \label{IAcase1b}
     \end{subfigure}
     \vspace{2 em}
     \begin{subfigure}[b]{0.45\textwidth}
         \centering
         \includegraphics[width=\textwidth]{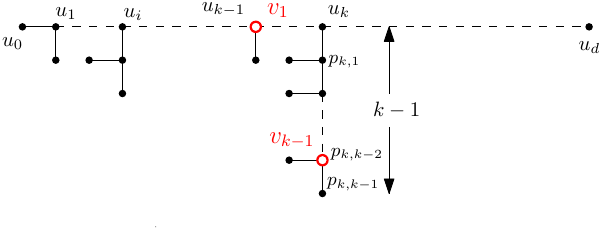}
         \caption{}
         \label{IAcase1c}
     \end{subfigure}
     \hfill
     \begin{subfigure}[b]{0.45\textwidth}
         \centering
         \includegraphics[width=\textwidth]{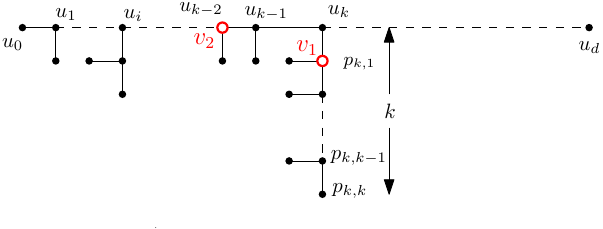}
         \caption{}
         \label{IAcase1d}
     \end{subfigure}
        \caption{Case 1}
        \label{case1}
\end{figure}
\\
b) Next, we consider the situation when $L_{k-1}$ is of height $1$, i.e., a single pendant is attached to $u_{k-1}$. We observe that $T(u_{k-1})\cup \{p_{k,1}\}$ have order $\geq 2k+2$.  

\vspace{0.7em}

If the branch $L_{k}$ attached to $u_{k}$  is of height at most $k-2$, then we place $v_{1}$ at $u_{k-1}$ to burn $T(u_{k-1})\cup 
\{p_{k,1}\}$.(see Figure \ref{IAcase1b}). 
\\
Again, if $L_{k}$ is of height $k$, then we place $v_{1}$ at $p_{k,1}$ and $v_{2}$ at $u_{k-2}$ to burn $\geq 4k$ vertices. In both of these situations,  we update the diametral path $P_{d+1}$ to $P_{d-k+1}=(u_{k},\hdots,u_{d})$ and the modified FBTNP becomes $T^{\prime}(u_{k+1})\cup \{u_{k}\}$. (Refer Figure \ref{IAcase1d})


\vspace{0.7em}

When $L_{k}$ is of height $k-1$ and furthermore there is a branch of height $\geq 2$ at $p_{k,1}$ different from $L_{k}$, then by placing $v_{1}$ at $p_{k,1}$ and $v_{2}$ at $u_{k-2}$ we burn $T(u_{k})\setminus \{u_{k}\}$ upto $k$ steps, i.e. total of $\geq 4k$ vertices, otherwise, we place $v_{1}$ at $u_{k-1}$ to burn $\geq 4k-5$ many vertices to burn $T(u_{k})\setminus T(p_{k, k-2})$ up to $k$ steps. Next, we place $v_{k-1}$ at $p_{k,k-2}$ to burn $T(p_{k,k-2})$. Therefore, total $\geq 4k-2$ vertices get burnt by burning $T(u_{k})\setminus \{u_{k}\}$ up to $k$ steps. Also, to burn the modified tree $T^{\prime}(u_{k+1})\cup \{u_{k}\}$ we use Algorithm \ref{Algo1}.
(Refer Figure \ref{IAcase1c})

 \vspace{0.7em}
\noindent \textbf{Case 2:} Next, we consider the case when $T(u_{k-1})$
have order less than $2k+1$. Since $k<h$, no vertex $u_{i}$, $0\leq i \leq k$  can be the root of $T$. Therefore, number of vertices of $T(u_{k-1})$ is $\geq 2k-1$. In fact, $|V(T(u_{k-1})|=2k-1$ as $T$ is an FBTNP. Hence the total number of vertices of $T(u_{k})\geq (2k-1)+2=2k+1$. The following situations may arise:
\vspace{0.7em}

\textbf{a)} First, we consider the situation when the branch $L_{k}$ attached to $u_{k}$ is of height $\geq 2$.

\vspace{0.7em}
\begin{figure}[h]
     \centering
     \begin{subfigure}[b]{0.4\textwidth}
         \centering
         \includegraphics[width=\textwidth]{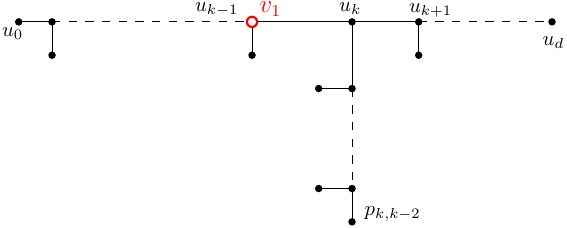}
         \caption{}
         \label{IAcase2a}
     \end{subfigure}
     \hfill
     \begin{subfigure}[b]{0.4\textwidth}
         \centering
         \includegraphics[width=\textwidth]{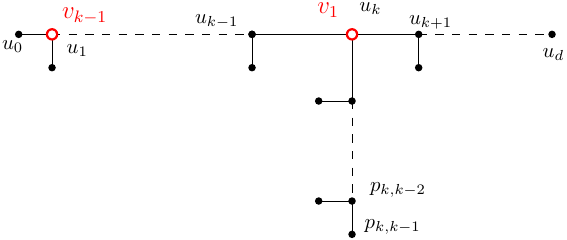}
         \caption{}
         \label{IAcase2b}
     \end{subfigure}
     \vspace{0.5 em}
     \begin{subfigure}[b]{0.4\textwidth}
         \centering
         \includegraphics[width=\textwidth]{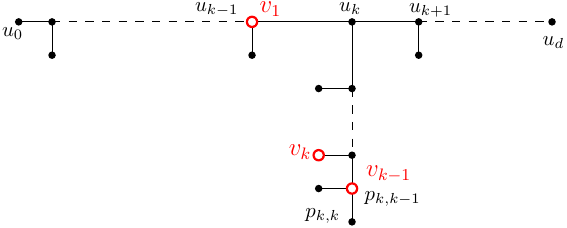}
         \caption{}
         \label{IAcase2c}
     \end{subfigure}
        \caption{Case 2a}
        \label{case1}
\end{figure}

If the branch $L_{k}$ attached to $u_{k}$ is of maximum height $k-2$, then order of $T(u_{k})\setminus \{u_{k}\}$ is $\geq (2k-1)+3=2k+2$. We place $v_{1}$ at $u_{k-1}$ to burn $\geq 2k+2$ vertices to burn $T(u_{k})\setminus \{u_{k}\}$ up to $k$ steps. Next, we update the diametral path $P_{d+1}$ to $P_{d-k+1}=(u_{k},\hdots,u_{d})$ so that the modified FBTNP becomes $T^{\prime} (u_{k+1})\cup \{u_{k}\}$. (Refer Figure \ref{IAcase2a}) 

\vspace{0.7em}

Now, consider the situation when  $L_{k}$ is of height $k$. If there is a branch of height $\geq 2$ from any $p_{k,i}$ where $1\leq i\leq k-2$, then $|V(T(p_{k,1})|\geq 2k+1$. We can modify the diametral path $P_{d+1}$ to $(p_{k,k},\hdots,p_{k,1}, u_{k}, \hdots, u_{d})$ and hence we can apply Case $1$ by taking $|V(T(p_{k,1})|\geq 2k+1$. In another situation, order of  $T(u_{k})\setminus \{u_{k}\}$ is $(2k-1)+(2k-1)=4k-2$.  Therefore, we place $v_{1}$ at $u_{k-1}$  and $v_{k-1}$ at $p_{k,k-1}$ and $v_{k}$ at the pendant attached to $p_{k,k-2}$ to burn $4k-2$ vertices  up to $k$ steps. (Refer Figure \ref{IAcase2c})

Again, if $L_{k}$ is of height $k-1$, we place $v_1$ at $u_{k}$ and $v_{k-1}$ at $u_{1}$ to burn $4k-4$ vertices of the subtree $T(u_{k})\setminus \{u_{k}\}$. In both of these circumstances, we update the diametral path to $P_{d-k+1}$ and to burn the modified FBTNP $T^{\prime}(u_{k+1})\cup \{u_{k}\}$ we apply Algorithm \ref{Algo1}. (Refer Figure \ref{IAcase2b})

\vspace{1em}
\textbf{b)} If there is only one pendant attached to $u_{k}$, i.e., $|V(T(u_{k})|=2k+1$. Let $u_{k+1}$ be the root of $T$. Then we place $v_{1}$ at $u_{k-1}$ to burn to burn $2k+2$ vertices of the subtree $T(u_{k+1})$ in $k$ steps. We update the diametral path $P_{d+1}$ to $L_{k+2}\cup(u_{k+2},\hdots,u_{d})$ and hence the modified FBTNP becomes $T^{\prime}(u_{k+2})$.

\begin{figure}[ht]
     \centering
     \begin{subfigure}[b]{0.4\textwidth}
         \centering
         \includegraphics[width=\textwidth]{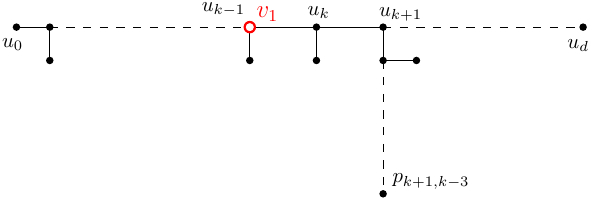}
         \caption{}
         \label{IAcase3a}
     \end{subfigure}
     \hfill
     \begin{subfigure}[b]{0.4\textwidth}
         \centering
         \includegraphics[width=\textwidth]{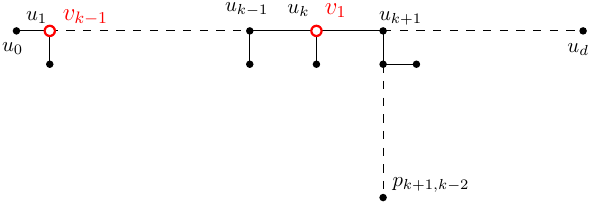}
         \caption{}
         \label{IAcase3b}
     \end{subfigure}
     \vspace{2 em}
     \begin{subfigure}[b]{0.4\textwidth}
         \centering
         \includegraphics[width=\textwidth]{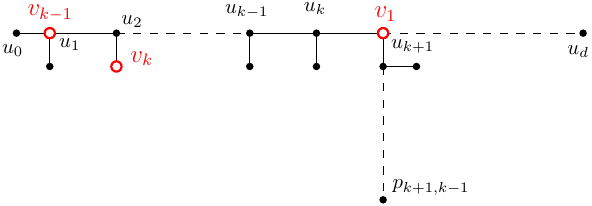}
         \caption{}
         \label{IAcase3c}
     \end{subfigure}
     \hfill
     \begin{subfigure}[b]{0.4\textwidth}
         \centering
         \includegraphics[width=\textwidth]{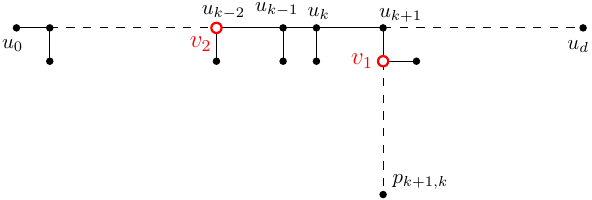}
         \caption{}
         \label{IAcase3d}
     \end{subfigure}
     \vspace{2 em}
     \begin{subfigure}[b]{0.4\textwidth}
         \centering
         \includegraphics[width=\textwidth]{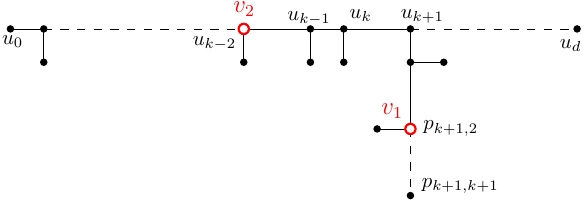}
         \caption{}
         \label{IAcase3e}
     \end{subfigure}
        \caption{Case 2b}
        \label{case3}
\end{figure}

Now, consider the situation when $u_{k+1}$ is not the root of $T$. Let $T(u_{k})\cup \{p_{k+1,1}\}$ has order $\geq 2k+2$. Since $k<h$, no vertex of the path $L_{k+1}$ can be the root of $T$.

    \vspace{0.7em}
    If the branch $L_{k+1}$ attached to $u_{k+1}$ is of height $\leq k-3$, then we place $v_{1}$ at $u_{k-1}$ to burn $T(u_{k+1})$ up to $k$ steps. Next, we update the diametral path $P_{d+1}$ to $P_{d-k}=(u_{k+1},\hdots,u_{d})$ so that the modified FBTNP becomes $T^{\prime}(u_{k+2})\cup \{u_{k+1}\}$. (Refer Figure \ref{IAcase3a})
    
\vspace{0.7em}
 
   If $L_{k+1}$ is of height $k+1$. Let the subtree $T(u_k) \cup T(p_{k+1,1})$ have order exactly equals $4k+2$.
   Then every branch from the internal nodes of $L_{k+1}$ is of height exactly equals $1$.  We place $v_{1}$ at $p_{k+1,2}$ and $v_{2}$ at $u_{k-2}$ to burn $T(u_{k+1})\setminus \{u_{k}\}$ up to $k$ steps. Hence the modified FBTNP is $T^{\prime}(u_{k+2})\cup \{u_{k+1}\}$. (Refer Figure \ref{IAcase3e})
   
   Again, if the subtree $T(u_{k})\cup T(p_{k+1,1})$ has its order $>4k+2$, then there must be some branch from the vertices of  $L_{k+1}$ of height $\geq 2$, different from $L_{k+1}$ as $|V(T(u_{k})|=2k+1$. 
    Now, if the order of the subtree $T(p_{k+1,2})\geq 2k+1$, then we update the diametral path $P_{d+1}$ to $L_{k+1}\cup (u_{k+1},\hdots, u_{d})$ and apply Case $1$, otherwise, i.e., when the order of the subtree $T(p_{{k+1},1})\geq 2k+3$ and $|V(T(p_{k+1,2}))|=2k-1$, then also we update the the diametral path $P_{d+1}$ to $L_{k+1}\cup (u_{k+1},\hdots, u_{d})$ and apply Case $2$(a).

   \vspace{0.7em}

    
    If the branch $L_{k+1}$ is of height $k$, then the order of the subtree $T(u_k) \cup T(p_{k+1,1})$ has order $\geq 4k$.
   We place $v_1$ at $p_{k+1,1}$ and $v_2$ at $u_{k-2}$ to burn $T(u_{k+1})\setminus \{u_{k+1}\}$ up to $k$ steps. After this, we update the diametral path $P_{d+1}$ to $P_{d-k}$ so that the modified tree $T^{\prime}(u_{k+2})\cup \{u_{k+1}\}$ preserves all properties of the initial tree $T$. (Refer Figure \ref{IAcase3d}) 


 \vspace{0.7em}

  When the branch $L_{k+1}$ is of height $k-1$, then if the order of the subtree $T(u_k) \cup T(p_{k+1,1})$ has order $
4k-2 (\geq 2k+2)$, then we place $v_{1}$ at $u_{k+1}$ and $v_{k-1}$ at $u_{1}$ and $v_{k}$ at the leaf of $u_{2}$ to burn $T(u_{k+1})\setminus \{u_{k}\}$ up to $k$ steps (Refer Figure \ref{IAcase3c}). After this placement, we update the diametral path to $P_{d-k}$ and the modified FBTNP becomes $T^{\prime}(u_{k+2})\cup \{u_{k+1}\}$.

 \vspace{0.7em}
  If the branch $L_{k+1}$ is of height $k-2$, then $T(u_{k})\cup T(p_{k+1,1})$  has order $\geq 4k-4$. We place $v_{1}$ at $u_{k}$ and $v_{k-1}$ at $u_{1}$ to burn $T(u_{k})\cup T(p_{k+1,1})$ up to $k$ steps (Refer Figure \ref{IAcase3b}). Next, we update the diametral path to $P_{d-k}$ so that the modified FBTNP becomes $T^{\prime} (u_{k+2})\cup \{u_{k+1}\}$.

\vspace{0.7em}
\noindent \textbf{Step II:} Repeat Step I.

\secondalgo*

\begin{proof}
Let $P_{d+1}=(u_{0},u_{1},\hdots,u_{d})$ be the diametral path of $T$. We follow the Algorithm \ref{Algo2} to burn $T$ in $k$ steps. We will use induction on $|V(T)|$ to prove the theorem.

\vspace{0.6em}
\noindent \textit{Induction Hypothesis :} Suppose, all FBTNP trees having order less than equal $(t+1)^{2}-9$, i.e. $|V(T)|\leq (t+1)^{2}-9$ can be burnt in $t$ steps.

\vspace{0.3em}
\noindent \textit{Inductive Step:} Consider a tree $T$ with $m$ vertices such that $(t+1)^2-9<m \leq (t+2)^{2}-9$. In order to prove the theorem, it is sufficient to show that $T$ can be burnt in $t+1$ steps. We set $k=t+1$.
For $k\geq h$, the verification is similar to what we did in Theorem \ref{firstalgo}.

\vspace{0.3em}
  Consider the case when $k<h$. In the first step of Algorithm \ref{Algo2}, different vertices have been chosen as sources of fire as per requirement. The choices are made among the following sources $\{v_{1},v_{2},v_{k-1},v_{k}\}$ before updating the diametral path of $T$ for the first time. We have burnt a certain number of vertices in each of the cases in such a manner so that the modified FBTNP remains connected and the remaining graph gets burnt completely up to $k$ steps. In each of the cases, we prove $k=\lceil \sqrt{n+9}\rceil -1$.

\vspace{0.3em}

 We select $v_{1}$ as a source of fire to burn $\geq 2k+2$ vertices for some cases. Therefore, the number of vertices burnt due to the fire spread from $v_{1}$ up to $t+1$ th step is $\geq 2t+4$. Therefore, the order of the modified tree becomes less than equal $m-(2t+4)$. Since $m\leq (t+2)^{2}-9$, the number of vertices of the modified tree is $\leq (t+2)^{2}-9-(2t+4)=(t+1)^{2}-10 < (t+1)^{2}-9$. Hence, it can be burnt in $t$ steps by induction hypothesis.

 \vspace{0.5em} 

      We select $v_{1}$ and $v_{2}$ at suitable places over the branches of $T$ to burn $\geq 4k$ vertices. Therefore, the number of vertices burnt up to $t+1$ th step from the fire spread from $v_{1},v_{2}$ is $\geq 4t+4$. Hence the order of the modified tree is $\leq m-(4t+4)\leq (t+2)^{2}-9-4t-4=t^{2}-9$. Therefore, the modified tree can be burnt in $t-1$ steps by induction hypothesis.

 \vspace{0.5em}

 The remaining cases have been proved with the help of Theorem \ref{firstalgo}.
\vspace{0.5em}
 
  We select $v_1$ and $v_{k-1}$ at appropriate locations over the branches of $T$ to burn $\geq 4k-4$ vertices. Therefore, the number of vertices burnt up to $t+1$ th step from the fire spread from $v_{1},v_{t}$ is $\geq 4t$. Hence the order of the modified tree is $\leq m-4t\leq (t+2)^{2}-9-4t=t^{2}-5<t^{2}$. 

Again, in some situations, to burn $4k-2$ vertices, we choose $v_{1},v_{k-1},v_{k}$ suitably at proper places on the branches of $T$. Hence, the number of vertices burnt up to $t+1$ th step from the fire spread from $v_{1},v_{t},v_{t+1}$ is $\geq 4t+2$.  Therefore, the order of the modified tree is $m-(4t+2)\leq (t+2)^{2}-9-4t=t^{2}-7<t^{2}$.
 = (

 \vspace{0.5em}
We see in both of the above cases the modified tree has order $\leq t^2 -5$. Now, we claim that $(v_2, \hdots, v_{t-2})$ is sufficient to burn the modified tree. From the proof of Theorem \ref{firstalgo}, we know, we can burn a tree of order $t^2$ in $t$ steps (i.e., $(v_1^{\prime}, \hdots, v_t^{\prime})$. Note that, the sources $v_{t-1}^{\prime}$ and $v_t^{\prime}$ burn a maximum of $5$ vertices in an FBTNP. Thus, using $(v_1^{\prime}, \hdots, v_{t-2}^{\prime})$ is enough to burn a tree using a tree of order $t^2 -5$. It is easy to observe that the burning sequence  $(v_{2},v_{3},\hdots v_{t+1})$ is equivalent to $(v_{1}^{\prime},v_{2}^{\prime},\hdots,v_{t}^{\prime})$.

 \vspace{0.5em} 
\noindent \textit{Conclusion:} Thus, by induction we are able to prove that an FBTNP $T$ having $|V(T)|\leq (t+2)^{2}-9$ can be burnt in $t+1$ steps. Hence we get $b(T)\leq \lceil \sqrt{n+9}\rceil-1$. 
\hfill $\square$ 
\end{proof}

\section{General Tree}\label{sec:generaltree}
We now extend our results to the more general tree. The $k$-ary tree is a rooted tree, where each node can hold at most $k$ number of children. We define a special type of $k$-ary tree, where each internal node can hold at least $2$ children and at most $k$ children, we name this tree as $(3,k)$-ary tree. 
Note that, Algorithm \ref{Algo2} works for this special kind of $k$-ary.

In 2018, Bessy et al. proved \cite{bessy2018bounds}  that $b(T) \leq \lceil \sqrt{n + n_2 + 1/4} +1/2 \rceil$ for a tree $T$ where $n_2$ is the number of degree $2$ vertices in $T$. Using Theorem \ref{secondalgo}, we improve this bound for the trees having the number of vertices at least $50$.
\generaltree*

\begin{proof}
   Let a tree $T$ of order $n$ has $n_2$ nodes of degree $2$. We add a pendent vertex to each of $n_2-1$ nodes leaving out $1$ node as the root, thus transforming $T$ to a $(3,k)$-ary $T^{\prime}$ of order $n+n_2-1$. We know that Algorithm \ref{Algo2} works for $(3,k)$-ary tree.
By Theorem \ref{generaltree} we can burn $T^{\prime}$ in $ \lceil \sqrt{(n+n_2-1)+9}\rceil -1 = \lceil \sqrt{n+n_2+8}\rceil -1$ steps. \hfill $\square$ 
\end{proof}
The burning number of a connected graph G equals the burning number of a spanning tree of G; see \cite{bonato2016burn}. Hence, we derive the following results for a connected graph. 
As, $n_2 \leq n-2$, thus a tree of order $n$ can be burnt in $\lceil \sqrt{2n + 6}\rceil -1$ steps which is an improvement over \cite{bessy2018bounds}.
Also, when $n_2 \leq n/3$, then $b(T) \leq \lceil \sqrt{4n/3+8}\rceil -1$ which is an improvement over \cite{bonato2021improved} under the given condition.

\section{Conclusion}
To sum up, the burning number conjecture has been proved for the 
perfect binary tree, complete binary tree, and  FBTNP. For an FBTNP, an improved algorithm has been given which in turn has improved the bound from the original conjecture.  

 The burning number conjecture could also be studied for the other interesting subclasses of trees.




\bibliography{ref.bib}






\end{document}